\documentclass[10pt,reqno]{amsart}
\usepackage{amsthm,amsmath,amssymb,amscd}
\usepackage{mathrsfs}
\usepackage{lscape}
\usepackage{graphics} 
\usepackage[all,cmtip]{xy}
\usepackage{graphicx}
\usepackage{subfigure}
\usepackage{url}	
\usepackage{hyperref}
\usepackage{color}
\usepackage[usenames,dvipsnames]{xcolor}
\usepackage{verbatim}
\usepackage{fancyhdr}
\usepackage{geometry}
\usepackage{cancel}
\usepackage[normalem]{ulem}
\usepackage{mathtools}
\usepackage[english]{babel}
\usepackage{enumerate}
\usepackage{enumitem}

\mathtoolsset{showonlyrefs,showmanualtags}
\newtheorem{thm}{Theorem}[section]

\newtheorem{prop}[thm]{Proposition}
\newtheorem{cor}[thm]{Corollary}
\newtheorem{lem}[thm]{Lemma}
\theoremstyle{definition}\newtheorem{defn}[thm]{Definition}
\theoremstyle{remark}\newtheorem{rmk}[thm]{Remark}
\newtheorem{eg}[thm]{Example}
\numberwithin{equation}{section}

\DeclareMathOperator \rank{rank}

\DeclareMathOperator \cdim{dim}
\DeclareMathOperator \cspan{span}
\def\iu {\sqrt{-1}}
\def\pbp {\partial\bar\partial}
\def\CP {\mathbb C\mathrm P}

\begin{document}
\title{Extensions of extremal K\"ahler submanifolds of complex projective spaces}
\author[Chao Li] {Chao Li}
\address{ICTP Trieste, cli1@ictp.it, leecryst@mail.ustc.edu.cn}

\thanks{The author is partially supported by NSFC, No. 12001512.}

\begin{abstract} 
In this paper we showed that every connected extremal K\"ahler submanifold of a complex projetive space has a natural extension which is a complete K\"ahler manifold and admits a holomorphic isometric immersion into the same ambient space. We also give an application to study extremal K\"ahler hypersurfaces of complex projective spaces.
\end{abstract}

\maketitle

\tableofcontents

\section{Introduction}

Extremal K\"ahler metrics were introduced by Calabi (\cite{Cal82}) as canonical metrics on K\"ahler manifolds. On a compact complex manifold, a K\"ahler metric is said extremal,  if in its K\"ahler class, it is a critical point of the functional $L^2$-norm of the scalar curvature. In this case, a K\"ahler metric is extremal if and only if the gradient of its scalar curvature is holomorphic. In the non-compact case, we also define  a K\"ahler metric to be extremal if the gradient of its scalar curvature is holomorphic. In this paper, a K\"ahler submanifold is said extremal if the induced K\"ahler metric is extremal.

One motivation of studying canonical metrics is to find generalizations of the classical uniformization  theorem in the setting of K\"ahler geometry. So a natural problem related is the classification of extremal K\"ahler submanifolds of complex space forms. Calabi (\cite{Cal53}) established a necessary and sufficient condition for a K\"ahler manifold to be locally holomorphically and  isometrically embedded in a given complex space form. Later on there have been many works on the classification of K\"ahler Einstein submanifolds of complex space forms  (see for example \cite{Smy67,Chern67,Kob67,Hano75,Tsu86,Ume87}). 
Loi-Salis-Zudda (\cite{LSZ21}) started the study of the classification of extremal K\"ahler submanifolds of complex space forms. They conjectured that these submanifolds must be locally homogeneous and gave an affirmative answer when the metric is radically symmetric under some local complex coordinate system (see \cite[Conjecture 1 and Theorem 1.1]{LSZ21}). There is still a long way to go before their conjecture is solved.

In this paper we are mainly interested in the extensions of non-closed extremal K\"ahler submanifolds of complex projective spaces $\CP^N$ endowed with canonical Fubini-Study metrics $\omega_{FS}$. Currently, it is not known whether these submanifolds can be extended to compact ones, but we are able to show that they have natural extensions which are complete K\"ahler manifolds and admit holomorphic isometric immersions into $(\CP^N,\omega_{FS})$. For the special case of K\"ahler Einstein submanifolds,  Hulin (\cite{Hul96}) has showed that they can be extended to injectively immersed K\"ahler submanifolds with complete induced metrics. 

Our first result is on general (non-closed) complex submanifolds. For convenience, we introduce
\begin{defn}
Let $M_0$ be a complex manifold. A pair $(M,\phi)$ , which consists of a connected $n$-dimensional complex manifold and a holomorphic immersion $\phi:M\rightarrow M_0$, is called an $n$-dimensional complex leaf pair of $M_0$ if the following universal property is satisfied
\begin{itemize}
\item[$(\star)$] For any holomorphic immersion $\phi_1:M_1\rightarrow M_0$ with $M_1$ a connected $n$-dimensional complex manifold, if $\phi_1(U_1)\subset \phi(M)$ for some nonempty open set of $M_1$, then there is a unique holomorphic map $\rho_{\phi_1}:M_1\rightarrow M$ such that $\phi_1=\phi\circ \rho_{\phi_1}$.
\end{itemize}

Let $S$ be a connected $n$-dimensional complex submanifold of $M_0$. A leaf extension of $S$ in $M_0$ is an $n$-dimensional complex leaf pair $(M,\phi)$ of $M_0$ with $\phi(M)\supset S$.
\end{defn}

A leaf extension $(M,\phi)$ of $S$ is indeed an extension of $S$ since the Definition implies there is an open holomorphic embedding $\rho_S:S\rightarrow M$ such that $\phi\circ\rho_S$ is the inclusion. Besides, the leaf extension is clearly unique up to equivalence. Here two pairs $(M_1,\phi_1)$ and $(M_2,\phi_2)$ are called equivalent if there exists a biholomorphic map $\rho:M_2\rightarrow M_1$ such that $\phi_2=\phi_1\circ \rho$.

By studying the space of germs of complex submanifolds (following the idea of classical analytic continuation),  we obtain
\begin{thm}\label{rle}
For any connected complex submanifold $S$ of a complex manifold $M_0$, there exists a leaf extension  of $S$ in $M_0$.
\end{thm}

Our main result is on the leaf extension of an extremal K\"ahler submanifold of $(\CP^N,\omega_{FS})$.
\begin{thm}\label{mthm1}
Let $(S,\omega_S)$ be a connected $n$-dimensional K\"ahler submanifold of $(\CP^N,\omega_{FS})$ and $(M,\phi)$ the leaf extension of $S$ in $\CP^N$. If $\omega_S$ is extremal, then $(M,\phi^*\omega_{FS})$ is complete.
\end{thm}

Theorem \ref{rle} and Theorem \ref{mthm1} allow us to treat extremal K\"ahler submanifolds of $(\CP^N,\omega_{FS})$  as open sets of complete K\"ahler manifolds which admits holomorphic isometric immersions into $(\CP^N,\omega_{FS})$. For instance, we can apply this trick to prove
\begin{thm}\label{thm2}
Let $(S,\omega_S)$ be a connected extremal  K\"ahler hypersurface of  $(\CP^{n+1},\omega_{FS})$. Then the scalar curvature $R(\omega_S)$ is non-constant only if $n\geq 3$ and $R(\omega)\leq 4n(n+1)-8$.
\end{thm}

Here we give an outline of this paper.  In Section \ref{prel}, we  introduce some basic results about Bochner coordinates, rigidity of holomorphic maps into complex Euclidean spaces and extensibility of holomoprhic maps which solve certain first order PDE systems. In Section \ref{sgcs}, we study the space of germs of complex submanifolds of a fixed dimension and give proof to Theorem \ref{rle}. In section \ref{pfmthm}, we give a short proof to Theorem \ref{mthm1} in the beginning and then prove some lemmas used in the proof. In section  \ref{lccs}, we find a special kind of local complex coordinate systems for a class of complete K\"ahler manifold and then give proof to Theorem \ref{thm2}.

\section{Preliminaries}\label{prel}

\subsection{Conventions}
On a complex projective space $\CP^N$, $\{w^0,\cdots,w^N\}$ denotes the standard homogeneous coordinate system. Using this notation we can formally write
\begin{equation}
\omega_{FS}={\textstyle\frac{\iu}{2}}\pbp\log |w|^2.
\end{equation}
Given an $A\in \mathfrak{gl}(N+1,\mathbb C)$, $X^{[A]}$ denotes the corresponding holomorphic vector field on $\CP^N$ to $[A]$ under the canonical equivalence $\mathfrak{gl}(N+1,\mathbb C)/\mathbb CI_{N+1}\simeq H^0(\CP^N,T^{1,0}\CP^N)$. $\phi^{[A]}:\mathbb C\times \CP^N\rightarrow\CP^N$ is the holomorphic flow generated by $X^{[A]}$. Namely
\begin{equation}
\phi^{[A]}(\tau,[w])=[e^{\tau A}w].
\end{equation}
Given a $[w_0]\in\CP^N$, $H_{[w_0]}$ denotes the hyperplane at the infinite of $[w_0]$. Namely
\begin{equation}
H_{[w_0]}=\{[w]\in\CP^N|w_0^*w=0\}.
\end{equation}

\smallskip

For $n\geq 1$ and $r\geq 0$, we use the notations
\begin{equation}
B^n_r=\{z\in\mathbb C^n||z|<r\},\qquad \bar B^n_r=\{z\in\mathbb C^n||z|\leq r\}.
\end{equation}

\subsection{Bochner coordinates}

The Bochner coordinate system is a useful tool in the study of  K\"ahler manifolds with real-analytic metrics.
This kind of coordinate systems were first construct by Bochner (\cite{Boc47}) and their uniqueness was proved by Calabi (\cite{Cal53}).
\begin{defn}
Let $(M,\omega)$ be a real-analytic K\"ahler manifold.
A Bochner coordinate system of $(M,\omega)$ at $p\in M$ is a coordinate system $\{z^1,\cdots,z^n\}$ around $p$ with $z(p)=0$, such that 
a local potential $\varphi$ of $\omega$ around $p$ can be expanded as
\begin{equation}
\varphi=\frac{1}{2}|z|^2+\sum_{|I|,|J|\geq 2}a_{I,J}z^I\bar z^J.
\end{equation}

\begin{eg}\label{bcCPN} A Bochner coordinate system  $\{z^1,\cdots,z^N\}$ of $(\CP^N,\omega_{FS})$ at $[w_0]\in\CP^N$ can be expressed as
\begin{equation}
(1,z)=\frac{Q^*w}{w_1^*w},
\end{equation}
where $Q=(Q^i_j)$ is an $(N+1)$-order unitary matrix whose first column vector $w_1$ lies in $[w_0]$.
Under the Bochner coordinate system, a potential of $\omega_{FS}$ is simply
\begin{equation}
{\textstyle\frac{1}{2}}\log (1+|z|^2).
\end{equation}
Clearly $z^1,\cdots,z^N$ are meromorphic on $\CP^N$ and holomorphic on $\CP^N\!\setminus\! H_{[w_0]}$. 
\end{eg}

\end{defn}
\begin{prop}[{\cite[page 181]{Boc47}},{\cite[page 14]{Cal53}}]
Let $(M,\omega)$ be a real-analytic K\"ahler manifold. Then for any $p\in M$, there exists a Bochner coordinate system of $(M,\omega)$ at $p$. 
Furthermore, the Bochner coordinate system of $(M,\omega)$ at $p$  is unique up to homogeneous unitary transposition.
\end{prop}

\begin{prop}[{\cite[Theorem 7]{Cal53}}]\label{bcim}
Let $\phi:M\rightarrow M_0$ be a holomorphic immersion from an $n$-dimension complex  manifold into an $N$-dimensional complex manifold. If $\omega_0$ is a real-analytic K\"ahler metric on $M_0$ and $\omega=\phi^*\omega_0$. Then for any $p\in M$, there exists Bochner coordinate systems $\{z^1,\cdots,z^n\}$ of $(M,\omega)$ at $p$ and $\{y^1,\cdots,y^{N}\}$ of $(M_0,\omega_0)$ at $\phi(p_0)$, such that
\begin{equation}
y^i\circ \phi=\left\{\begin{array}{ll}
z^i,\qquad &i=1,\cdots,n,\\
f^i(z),\qquad &i=n+1,\cdots,N,
\end{array}
\right.
\end{equation}
where $f^i(0)=0$  and $Df^i(0)=0$ for $i=n+1,\cdots,N$.
\end{prop}

As mentioned in Introduction, we mainly concern the Bochner coordinate system of a K\"ahler manifold which can be locally holomorphically embedded in a complex projective space.
Using Proposition \ref{bcim} and Example \ref{bcCPN}, one can easily check that
\begin{prop}\label{imlc}
Let $(M,\omega)$ be an $n$-dimensional K\"ahler manifold and $\phi:(M,\omega)\rightarrow(\CP^N,\omega_{FS})$ a holomorphic isometric immersion.
For any $p\in M$, there exists an open neighborhood $U$ of $p$ with a Bochner coordinate system $\{z^1,\cdots,z^n\}$ at $p$ defined on $U$, such that
\begin{equation}
\phi=[Q(1,z,f)],
\end{equation}
where $Q$ is an $(N+1)$-order unitary matrix and $f:U\rightarrow \mathbb C^{N-n}$ is a holomorphic map with $f(p)=0$ and $\nabla f(p)=0$.
\end{prop}

\subsection{Rigidity of the holomorphic maps into complex Euclidean spaces}
 Let $M$ be a connected complex manifold of dimension $n$. A holomorphic map $F=(f^1,\cdots,f^m):M\rightarrow \mathbb C^m$ is said linearly full if its components $f^1,\cdots,f^m$ are linearly independent.

We have the following
\begin{lem}\label{rgdhol}
Let $F=(f^1,\cdots,f^m)$ and $G=(g^1,\cdots,g^m)$  be holomorphic maps from $M$ into $\mathbb C^m$.
\begin{enumerate}
\item If $\dim\cspan\{f^1,\cdots,f^m\}=k$ with $1\leq k<m$, then there exists a unitary matrix $Q$ and an linearly full holomorphic map $\tilde F: M\rightarrow \mathbb C^k$, such that $F=Q(\tilde F,0)$.
\item If $|F|^2=|G|^2$, then there exists an $m$-order unitary matrix $Q$, such that $G=QF$.
\item If $F^*G$ is real-valued and $F$ is linearly full, then there exists an $m$-order Hermitian matrix $H$, such that $G=HF$.
\end{enumerate}
\end{lem}
The idea of the proof is inspired by Calabi (\cite[Theorem 2]{Cal53}).
\begin{proof}
(1) Consider the subspace of $\mathbb C^m$
\begin{equation}
V=\{v\in\mathbb C^m|v^*F\equiv0\},
\end{equation}
Clear $\dim V=m-\dim\cspan\{f^1,\cdots,f^m\}=m-k$. If we choose an $m$-order unitary matrix $Q$ whose last $m-k$ column vectors span $V$, then clearly $Q^*F=(\tilde F,0)$ for some holomorphic map $\tilde F=(\tilde f^1,\cdots,\tilde f^k):M\rightarrow \mathbb C^k$ and $|F|^2=|\tilde F|^2$. As $Q$ is nonsingular, we have
\begin{equation}
\dim \{\tilde f^1,\cdots,\tilde f^k\}=\dim\{f^1,\cdots,f^m\}=k,
\end{equation}
So $\tilde F$ is full.

\smallskip

(2) According to (1), we only need to consider the case that $F$ is linearly full. First we claim that for any $(p,q)\in M\times M$
\begin{equation}\label{rgdholeq1}
F(q)^*F(p)=G(q)^*G(p).
\end{equation}
In fact, both the two sides of \eqref{rgdholeq1} are real-analytic functions on $M\times M$ which are holomorphic in $p$ and anti-holomorphic in $q$. At the same time, $|F|^2=|G|^2$ implies \eqref{rgdholeq1} holds on the diagonal of $M\times M$. By choosing a complex coordinate system $\{z^1,\cdots,z^n\}$ centering at some $p_0\in M$ and comparing the Taylor series  of the two functions in $(z,\bar z)$ at $(p_0,p_0)$, one can easily check that \eqref{rgdholeq1} holds on $U\times U$ with $U$ a neighborhood of $p_0$. Then by the connectedness of  $M$ and the uniqueness of real-analytic functions, \eqref{rgdholeq1} holds on $M\times M$.

Since $F$ is linearly full, we can find $p_1,\cdots,p_m\in M$ such that $F(p_1),\cdots,F(p_m)$ are linearly independent and consequently form a basis of $\mathbb C^m$. Consider the two $m$-order matrix
\begin{equation}
A=(F^i(p_j)),\qquad B=(G^i(p_j)).
\end{equation}
Clearly $A$ is nonsingular. Furthermore, by \eqref{rgdholeq1}, we have
\begin{equation}
A^*A=B^*B,\qquad F^*A=G^*B.
\end{equation}
As a result, we have
\begin{equation}
(BA^{-1})^*BA^{-1}=I_n,
\end{equation}
and
\begin{equation}
F=(A^{-1})^*A^*F=(A^{-1})^*B^*G=(BA^{-1})^*G.
\end{equation}
Let $Q=BA^{-1}$, then $Q$ is unitary and $G=QF$.

\smallskip

(3) The proof is similar to (2). Since $F^*G$ is real-valued, we have $F^*G=G^*F$. Following the proof of \eqref{rgdholeq1}, we can easily check that for any $(p,q)\in M\times M$
\begin{equation}\label{rgdholeq2}
F(q)^*G(p)=G(q)^*F(p).
\end{equation}
Since $F$ is linearly full, we can find $p_1,\cdots,p_m\in M$ such that $F(p_1),\cdots,F(p_m)$ are linearly independent and consequently form a basis of $\mathbb C^m$. Consider the two $m$-order matrix
\begin{equation}
A=(F^i(p_j)),\qquad B=(G^i(p_j)).
\end{equation}
Clearly $A$ is nonsingular. Furthermore, by \eqref{rgdholeq2}, we have
\begin{equation}
A^*B=B^*A,\qquad F^*B=G^*A.
\end{equation}
As a result, we have
\begin{equation}
(BA^{-1})^*=BA^{-1},
\end{equation}
and
\begin{equation}
G=(A^{-1})^*A^*G=(A^{-1})^*B^*F=(BA^{-1})^*F.
\end{equation}
Let $H=BA^{-1}$, then $H$ is Hermitian and $G=HF$.
\end{proof}

\subsection{Extensibility of holomoprhic maps which solve certain first order PDE systems}
\begin{prop}\label{extholo0}
Let $\Omega$ be a complex manifold and $V_1,\cdots,V_n$ global holomorphic vector fields on $\Omega$. If $r>0$ and $\xi:B^n_r\rightarrow \Omega$ is a holomorphic map such that
\begin{enumerate}
\item $\frac{\partial}{\partial z^i}$ is $\xi$-related with $V_i$ for each $i=1,\cdots,n$,.
\item $\xi(B_r)$ is relatively compact in $\Omega$.
\end{enumerate}
Then $\xi$ can be extended to a holomorphic map on $B^n_{\tilde r}$ for some $\tilde r>r$.
\end{prop}
Particularly in Section \ref{pfmthm} we need to use the follow special case of Proposition \ref{extholo0}
\begin{prop}\label{extholo}
Let $\Omega$ be an open domain in $\mathbb C^m$ and $\mathcal A:\Omega\rightarrow \mathbb C^m\otimes \mathbb C^n$ be a holomorphic map. If $r>0$ and $\xi:B^n_r\rightarrow \Omega$ is a holomorphic map such that
$D\xi=\mathcal A(\xi)$ and $X(B^n_r)$ is relatively compact in $\Omega$. Then $\xi$ can be extended to a holomorphic map on $B^n_{\tilde r}$ for some $\tilde r>r$.
\end{prop}
\begin{proof}[Proof of Proposition \ref{extholo0}]
Let $\tau$ be the standard complex coordinate on $\mathbb C$ and $\Delta_{\varepsilon}=\{\tau\in\mathbb C||\tau|<\varepsilon\}$ for any $\varepsilon>0$. Since $\xi(B_r)$ is relatively compact in $\Omega$, we can find  relatively compact domain $\Omega_0$ in $\Omega$ such that $\Omega_0\supset \overline{\xi(B_r)}$. Noting that each $V_j$ is holomorphic, we can find $\varepsilon_0>0$ and local holomorphic flow $\psi^j:\Delta_{\varepsilon_0}\times \Omega_0\rightarrow \Omega$ such that
\begin{equation}
d\psi^j_{(\tau,q)}\frac{\partial}{\partial\tau}\Big|_{(\tau,q)}=V_j\big|_{\psi^j(\tau,q)}.
\end{equation}
For each permutation $\sigma$ of $\{1,\cdots,n\}$ and $q\in\Omega_0$, we define
\begin{equation}
\psi^\sigma_q(z^1,\cdots,z^n)=\psi^{\sigma(1)}_{z^{\sigma(1)}}\circ\cdots\circ\psi^{\sigma(n)}_{z^{\sigma(n)}}(q).
\end{equation}
where $\psi^j_\tau(q)=\psi^j(\tau,q)$ for any $(\tau,q)\in \Delta_{\varepsilon_0}\times \Omega_0$. We can choose a small $\varepsilon_1>0$ such that each $\psi^{\sigma}_q$ is well-defined on $B^n_{\varepsilon_1}$.

By our construction, for any fixed $q\in \Omega_0$, the following two conditions are equivalent
\begin{enumerate}[label=(\alph*)]
\item For any two permutation $\sigma_1$ and $\sigma_2$, $\psi_q^{\sigma_1}=\psi_q^{\sigma_2}$.
\item There exists a holomorphic map $\psi:B^n_{\varepsilon_2}\rightarrow \Omega$ with some $\varepsilon_2>0$, such that $\frac{\partial}{\partial z^i}$ is $\psi$-related with $V_i$ for each $i=1,\cdots,n$.
\end{enumerate}
Furthermore, if Condition (b) holds, then for any permutation $\sigma$, $\psi_q^{\sigma}=\psi$  holds on $B^n_{\varepsilon_3}$ with $\varepsilon_3=\min\{\varepsilon_1,\varepsilon_2\}$.

Clearly for any $p\in B^n_r$, Condition (b) holds for $q=\xi(p)$ and consequently $\xi(z)=\psi_{\xi(p)}^{\sigma}(z-p)$ when $z\in B^n_{\varepsilon_1}\cap B(p,\varepsilon_2)$ for any permutation $\sigma$. Furthermore, since $B^n_r$ is convex, for any $p_1,p_2\in B^n_r$ with $W=B(p_1,\varepsilon_2)\cap B(p_2,\varepsilon_2)\neq\emptyset$ , $B^n_r\cap W$ is nonempty. Then $\psi_{\xi(p_1)}^{\sigma}(z-p_1)=\psi_{\xi(p_1)}^{\sigma}(z-p_2)$ for any $z\in W\cap B^n_r$ and consequently for any $z\in W$. By these facts we can easily verify that $\xi$ can be extended to a holomorphic map on $B^n_{r+\varepsilon_2}$.
\end{proof}

\section{The space of germs of complex submanifolds of a fixed dimension}\label{sgcs}
Let $M_0$ be a complex manifold and $1\leq n\leq \cdim M_0$. We start with some notations and definitions.
\begin{enumerate}[label=(\Alph*)]
\item $\mathcal V^n$ denotes the set of $n$-dimensional complex submanifolds of $M_0$.
\item For any $x\in M_0$,  the space of germs of $n$-dimensional complex submanifolds of $M_0$ at $x$ is
\begin{equation}
\mathcal S^n_x=\{S\in\mathcal V^n|p\in S\}/\sim_x,
\end{equation}
where $S_1\sim_x S_2$ if there exists a neighborhood $W$ of $x$ in $M_0$ such that $S_1\cap W=S_2\cap W$.
\item For any $x\in M_0$ and $S\in\mathcal V^n$ with $x\in S$, the germ  of $S$ at $x$ is the equivalent class
\begin{equation}
[S]_x=\{S_1\in\mathcal V^n|S_1\ni x \text{ and } S_1\sim_x S\}.
\end{equation}
\item $\mathcal S^n=\bigcup_{x\in M_0}\mathcal S^n_x$ is the {\bf  space of germs of $n$-dimensional complex submanifolds} of $M_0$, and $\sigma:\mathcal S^n\rightarrow M_0$ is the natural map $s\in\mathcal S^n_x\mapsto x$. 
\item For any $S\in\mathcal V^n$, $\rho_S:S\rightarrow \mathcal S^n$ denotes the natural map $x\mapsto [S]_x$ and $V_S=\{[S]_x|x\in S\}$.
\item The canonical topology of $\mathcal S^n_x$ is generated by the topology basis $\{V_S|S\in\mathcal V^n\}$.
\item A leaf $\mathcal L$ of $\mathcal S^n$ is a connected component of it.
\end{enumerate}

We have the following Theorem
\begin{thm}\label{thm0}
Up to equivalence, $(\mathcal S^n,\sigma)$ is the unique pair of topology space and continous map into $M_0$ such that
\begin{enumerate}
\item Each leaf $\mathcal L$ of $\mathcal S^n(M_0)$ admits the structure of an $n$-dimensional complex manifold such that $\sigma|_{\mathcal L}$ is a holomorphic immersion.
\item For any holomorphic immersion $\phi:M\rightarrow M_0$ with $M$ an $n$-dimensional complex manifold, there is a unique continuous map $\rho_{\phi}:M\rightarrow \mathcal S^n$ such that $\phi=\sigma\circ\rho_{\phi}$.
\end{enumerate}
\end{thm}
\begin{proof}
The uniqueness up to equivalence can be easily check by part (1) and (2). At the same time, part (2) can be easily checked by definition of $\mathcal S^n$ and its topology. Especially the map in (2) is $p\mapsto [\phi(U_p)]_{\phi(p)}$, where $U_p$ is an arbitrary open neighborhood of $p$ such that $\phi|_{U_p}$ is an embedding.

\smallskip

For part (1), we want to emphasize that in this paper, a complex manifold is always Hausdorff and secondly-countable. However, for non-secondly-countable spaces, we also use generalized concepts like complex atlas, holomoprhic map and Riemannian metric.

First we verify that $\mathcal S^n$ is Hausdorff. We only need to show that if $[S_1]_x\neq[S_2]_x$  for some $S_1,S_2\in \mathcal V^n$ and $x\in S_1\cap S_2$, then  there exist open subsets $V_1$ of $S_1$ and $V_2$  of $S_2$, such that $x\in V_1\cap V_2$ and $[S_1]_y\neq [S_2]_y$ for any $y\in V_1\cap V_2$. In fact, by the local properties of complex submanifolds, we  can find a connected open neighborhood $W$ of $x$ in $M_0$, such that
\begin{enumerate}[label=(\alph*)]
\item $V_1=S_1\cap W$ and $V_2=S_2\cap W$ are connected.
\item There exists holomorphic functions $f_1^{n+1},\cdots,f_1^{N}$ and $f_2^{n+1},\cdots,f_2^{N}$ on $W$, such that
\begin{equation}
V_1=\{f_1^{n+1}=\cdots=f_1^{N}=0\},\qquad V_2=\{f_2^{n+1}=\cdots=f_2^{N}=0\}.
\end{equation}
\end{enumerate}
Assume the opposite that $[S_1]_y=[S_2]_y$ for some $y\in V_1\cap V_2$, then $V_1\cap W_y=V_2\cap W_y$ for some open neighborhood $W_y$ of $y$ in $M_0$. Then by (a) and (b), we have $V_1=V_2$, which is contradict to the condition that $[S_1]_x\neq [S_2]_x$. So we obtain that $[S_1]_y\neq [S_2]_y$ for any $y\in V_1\cap V_2$.

Second we endow $\mathcal S^n$ with a complex atlas. Let $\mathcal A$ be the set of  pairs $(U,\varphi)$ which consists of an open set $U$ of $\mathcal S^n$ together with a map $\varphi:U\rightarrow \mathbb C^n$ such that there exists an $S\in\mathcal V^n$ with a holomorphic chart map $\varphi_S:S\rightarrow\mathbb C^n$ such that $U=U_S$ and $\varphi=\varphi_S\circ\rho^{-1}$.
Clearly each $(U,\varphi)\in\mathcal A$ is a chart and $\bigcup_{(U,\varphi)\in\mathcal A}U=M_0$. Furthermore, by the uniqueness of the complex structure of a complex submanifold, one can easily check any two charts in $\mathcal A$ are holomorphically compatible. So $\mathcal A$ is a complex atlas of $\mathcal S^n$

Third we show that a leaf $\mathcal L$ is secondly countable thus indeed a connected $n$-diemsnional complex manifold. Clearly $\phi$ is a holomorphic immersion. By pull-backing an arbitrary Riemannian metric on $M_0$, we can endow $M$ with a Riemannian metric $g$. As $M$ is Hausdorff, locally Euclidean and connected, $g$ induces a metric (distance) which is compatible with the topology of $M$. Furthermore, fixing a $p\in M$, any  $q\in M$ can be jointed with $p$ by finite many geodesic segments. By this fact one can easily check that $M$ is separable and consequently secondly countable.
\end{proof}

The following Corollary characterizes complex leaf pairs of $M_0$. We omit the proof.
\begin{cor}\label{leaf}
Let $\phi:M\rightarrow M_0$ be a holomoprhic immersion with $M$ an $n$-dimensional complex manifold. Then $(M,\phi)$ is  a complex leaf pair of $M_0$ if and only if it is equivalent to a pair $(\mathcal L,\sigma|_{\mathcal L})$ for some leaf $\mathcal L$ of $\mathcal S^n$.
\end{cor}

\begin{proof}[Proof of Theorem \ref{rle}]
Let $\iota:S\rightarrow M_0$ be the inclusion map  and $\mathcal L$ the leaf of $\mathcal S^n$ which contains a (consequently every)  germ of $S$. By Corollary \ref{leaf}, $(\mathcal L,\sigma|_{\mathcal L})$ is the leaf extension of $S$ in $\mathbb C^2$.
\end{proof}

\begin{eg}
Let $\alpha>0$ and $S_\alpha=\{(e^{\tau},e^{\alpha\tau})|\tau\in \mathbb C \text{ and }|\tau|<1\}\subset \mathbb C^2$. The following table shows the leaf extension $(M_\alpha,\phi_\alpha)$  of $S_\alpha$
\begin{center}
\begin{tabular}{|c|c|c|}
\hline
$\alpha$ & $M_\alpha$ & $\phi_\alpha$\\
\hline
$k\in\mathbb Z_{>0}$ & $\{z_1^k=z_2\}$ & the inclusion map\\
\hline
$\frac{1}{k}$ with $k\in\mathbb Z_{>0}$ & $\{z_1=z_2^k\}$ & the inclusion map\\
\hline
$\frac{b}{a}$ with primitive integers $a,b\geq 2$ & $\{z_1^b=z_2^a\}\!\setminus\!\{0\}$ & the inclusion map\\
\hline
irrational & $\mathbb C$ & $\phi_{\alpha}(\tau)=(e^{\tau},e^{\alpha\tau})$\\
\hline
\end{tabular}
\end{center}
\end{eg}

For later use, we introduce the following Propositions
\begin{prop}\label{cptleaf}
Let $(M_0,\omega_0)$ be a compact Hermitian manifold, $Z\subset M_0$ an irreducible analytic subset of dimension $n$ , $Z_{reg}$ the regular set of $Z$ and $(M,\phi)$ the leaf extension of $Z_{reg}$ in $M$. If $(M,\phi^*\omega_0)$ is complete, then $M$ is compact and consequently $\phi(M)=Z$.
\end{prop}
\begin{proof}
Let $Z_{sing}$ be the singular set of $Z$ and  $\Sigma=\phi^{-1}(Z_{sing})$. Clearly $\Sigma$ is an analytic subset of $M$. By the definition of leaf extensions one can check that $\phi:M\!\setminus\!\Sigma\rightarrow Z_{reg}$ is biholomorphic. Then
\begin{equation}
\mathrm{diam}(M,\phi^*\omega_0)=\mathrm{diam}(M\!\setminus\!\Sigma,\phi^*\omega_0)=\mathrm{diam}(Z_{reg},\omega_0|_{Z_{reg}})<\infty.
\end{equation}
Together with its completeness,  we deduce that $(M,\omega)$  is compact.
\end{proof}

\begin{prop}\label{qnleaf}
Let $(M,\phi)$ be the leaf extension of the following hypersurface in $\CP^{n+1}$
\begin{equation}
\{[1,z,f(z)]\in\CP^{n+1}|z\in B^n_r\},
\end{equation}
where $r>0$ and $f$ is a nonzero holomorphic function on $B^n_r$ with $f(0)=0$ and $Df(0)=0$. If $(M,\phi^*\omega_{FS})$ is complete and $f$ is a polynomial, then up to an automorphism of $\CP^{n+1}$, $\phi(M)$ is
\begin{equation}
Q_n=\{[w]\in\CP^{n+1}|(w^0)^2+(w^1)^2+\cdots+(w^n)^2+(w^{n+1})^2=0\}.
\end{equation}
Consequently $\phi$ is an embedding.
\end{prop}

\begin{proof}
Let $d=\deg(f)$ and $P$ be the homogeneous holomorphic polynomial of degree $d$ on $\mathbb C^{n+1}$ such that $P(1,z)=f(z)$ for any $z\in \mathbb C^n$. Consider
\begin{equation}
Z=\{[w]\in\CP^{n+1}|(w^0)^{d-1}w^{n+1}-P(w^0,w^1,\cdots,w^n)=0\}.
\end{equation}
Clearly $Z$ is irreducible and $(M,\phi)$ is also the leaf extension of the regular set of $Z$ in $\CP^N$. By Proposition \ref{cptleaf}, $M$ is compact and $\phi(M)=Z$. To finish the proof, we need to show that up to a homogeneous linear transposition
\begin{equation}
f(z)=(z^1)^2+\cdots+(z^n)^2.
\end{equation}

First we prove that $d=2$. Clearly $[0,\cdots,0,1]\in Z$. We choose a $p\in M$ such that $\phi(p)=[0,\cdots,0,1]$. For $i=0,1,\cdots,n$, consider the function defined around $p$
\begin{equation}
x^i=\phi\circ\frac{w^i}{w^{n+1}}.
\end{equation}
Obviously $x^i$ is holomorphic, $x^i(p)=0$ and
\begin{equation}\label{eqatinfty}
(x^0)^{d-1}=P(x^0,\dots,x^n),
\end{equation}
Since $\phi$ is an immersion, there is an $k\in\{0,1,\cdots,n\}$, such that $\{x^i|i\neq k\}$ is a complex coordinate system around $p$. Since $P$ is homogeneous of degree $d$, the vanishing order of $(x^0)^{d-1}=P(x^0,\cdots,x^n)$ at $p$ is at least $d$. So the vanishing order of $x^0$ is at least $2$ and consequently $\{x^1,\cdots,x^n\}$ is a complex coordinate system around $p$. By the choice of $P$, the vanishing order of $(x^0)^{d-1}=P(x^0,\cdots,x^n)$ at $p$ is exactly $d$. Therefore we must have $d=2$.

Since $d=2$, we can write
\begin{equation}
f(z)=\sum_{i,j=1}^n a_{ij}z^iz^j,
\end{equation}
where $(a_{ij})$ is a symmetric $n$-order  matrix. Let $n_0=\rank(a_{ij})$, then $1\leq n_0\leq n$. Up to a homogeneous linear transposition, we can write
\begin{equation}
f(z)=(z^1)^2+\cdots+(z^{n_0})^2,
\end{equation}
and then
\begin{equation}
Z=\{[w]\in\CP^{n+1}|w^0w^{n+1}-(w^1)^2-\cdots-(w^{n_0})^2=0\}.
\end{equation}
In the following we prove that $n_0=n$ by contradiction. Suppose that $n_0<n$, then $[0,\cdots,0,1,0]\in Z$. We choose a point $q\in M$ with $\phi(q)=[0,\cdots,0,1,0]$. For $i=1,\cdots,n-1,n+1$, consider the function defined around $q$
\begin{equation}
y^i=\phi\circ\frac{w^i}{w^{n+1}},
\end{equation}
Obviously $y^i$ is holomorphic, $y^i(q)=0$ and
\begin{equation}\label{eqatinfty1}
y^0 y^{n+1}=(y^1)^2+\cdots+ (y^{n_0})^2.
\end{equation}
$\phi$ is an immersion, so there should exist an $l\in\{0,1,\cdots,n-1,n+1\}$, such that $\{y^i|i\neq n,l\}$ is a complex coordinate system around $q$. However, this doesn't hold. If $l=0$,  we should have $y^{0}=(y^{n+1})^{-1}((y^1)^2+\cdots+ (y^{n_0})^2)$, while the right hand side is not holomorphic around $q$. If $l>n_0$, we should have  $y^{n+1}=(y^0)^{-1}((y^1)^2+\cdots+ (y^{n_0})^2)$, while the right hand side is not holomorphic around $q$. If $1\leq l\leq n_0$, we should have $(y^{l})^2=y^0y^{n+1}-(y^1)^2-\cdots-(y^{l-1})^2-(y^{l+1})^2-\cdots-(y^{n_0})^2$, while the right hand side is not the square of any holomorphic function around $q$.
\end{proof}

\section{Proof of the main theorem}\label{pfmthm}

In order to prove Theorem \ref{mthm1}, we need the following two Lemmas. Their proofs are put behind.
\begin{lem}\label{sc4}
Let $(M, \omega)$ be a connected $n$-dimensional K\"ahler manifold which admits a holomorphic isometric immersion $\phi$ into $ (\CP^N,\omega_{FS})$. Then $\omega$ is extremal if and only if there exists a holomorphic vector field $X$ and an $(N+1)$-order Hermitian matrix $A$, such that $X$ is $\phi$-related to $X^{[A]}$ and
\begin{equation}
n(n+1)-\frac{1}{4}R(\omega)=\frac{w^*Aw}{|w|^2}\circ \phi.
\end{equation}
\end{lem}
\begin{rmk}
In fact, by Theorem \ref{mthm1} and Lemma \ref{exstcrt} one can easily check that the matrix $A$ in Lemma \ref{sc4} can be chosen to be positive semi-definite.
\end{rmk}
\begin{lem}\label{lcext}
Let $N>n\geq 1$, $r>0$, $\Lambda\geq 0$ and $r_\Lambda=\big((3n+2)(\Lambda+1)^{\frac{1}{2}}\big)^{-1}$. If $f:B^n_r\rightarrow \mathbb C^{N-n}$ is a holomorphic map with $f(0)=0$ and $Df(0)=0$ such that
\begin{equation}\label{sceqA1}
n(n+1)-\frac{1}{4}R(\omega_f)=\frac{(1,z,f)^*A(1,z,f)}{|w|^2},
\end{equation}
where $\omega_f=\frac{\iu}{2}\pbp\log (1+|z|^2+|f|^2)$ and $A\leq \Lambda I_{N+1}$ is an $(N+1)$-order Hermitian matrix. Then
\begin{equation}\label{exteq1}
1+|z|^2+|f|^2+|Df|^2\leq (1-r_{\Lambda}^{-1}|z|)^{\frac{2}{3n+2}}.
\end{equation}
Moreover, if $r<r_{\Lambda}$, then $f$ can be extended to a holomorphic map on $B^n_{r_\Lambda}$.
\end{lem}

If $(M,\phi)$ is the leaf extension of a connected extremal K\"ahler submanifold of $(\CP^N,\omega_{FS})$, then $\phi^*\omega_{FS}$ is again extremal. By Lemma \ref{sc4},  Theorem \ref{mthm1} is a special case of the following  Theorem
\begin{thm}\label{mthm2}
Let $(M,\phi)$ be an $n$-dimensional complex leaf pair of $\CP^N$ and $\omega=\phi^*\omega_{FS}$. If  there exists an $(N+1)$-order Hermitian matrix $A$, such that
\begin{equation}\label{sceqA}
n(n+1)-\frac{1}{4}R(\omega)=\frac{w^*Aw}{|w|^2}\circ \phi.
\end{equation}
Then $(M,\omega)$ is complete.
\end{thm}
\begin{proof}
Applying  Proposition \ref{imlc}, for any $p\in M$, we can find a small $r>0$ and an open holomorphic embedding $\rho:B^n_r\rightarrow M$ such that $\rho(0)=p$ and
\begin{equation}
\phi\circ\rho(z)=[Q(1,z,f(z))],
\end{equation}
where $Q$ is $(N+1)$-order unitary matrix and $f:B^n_r\rightarrow \mathbb C^{N-n}$ is a holomorphic map with $f(0)=0$ and $Df(0)=0$. By \eqref{sceqA}, we have
\begin{equation}\label{sceqA2}
n(n+1)-\frac{1}{4}R(\omega_f)=\frac{(1,z,f)^*Q^*AQ(1,z,f)}{1+|z|^2+|f|^2},
\end{equation}
where $\omega_f=\rho^*\omega=\frac{\iu}{2}\pbp\log(1+|z|^2+|f|^2)$.

Let $\Lambda$ be the biggest eigenvalue of $A$. By \eqref{sceqA} and the easy fact $R(\omega)\leq 4n(n+1)$, $\Lambda\geq 0$. As $Q$ is unitary, the biggest eigenvalue of $Q^*AQ$ is again $\Lambda$. By Lemma \ref{lcext}, $f$ can be extended to a holomorphic map (also denoted by $f$) on $B^n_{r_\Lambda}$ with $r_\Lambda=\big((3n+2)(\Lambda+1)^{\frac{1}{2}}\big)^{-1}$, which satisfies
\begin{equation}\label{estA}
1+|z|^2+|f|^2+|Df|^2\leq (1-r_{\Lambda}^{-1}|z|)^{\frac{2}{3n+2}}.
\end{equation}

Consider the holomorphic embedding $\eta:B^n_{r_\Lambda}\rightarrow \CP^N$ defined as
\begin{equation}
\eta(z)=[Q(1,z,f(z))].
\end{equation}
By the definition of complex leaf pairs  and the equality $\eta|_{B^n_r}=\phi\circ\rho$, we can easily check that $\rho$ can be extended to a holomoprhic map (also denoted by $\rho$) form $B^n_{r_\Lambda}$ into $M$ such that $\eta=\phi\circ\rho$.
Let $U=\rho(B^n_{r_\Lambda})$, then $\rho:(B^n_{r_\Lambda},\omega_f)\rightarrow (U,\omega)$ is a holomorphic isometry. By \eqref{estA}, on $B^n_{\frac{1}{2}r_{\Lambda}}$
\begin{equation}
{\textstyle\frac{1}{2}}\omega_{eucl}\leq \omega_f\leq 2\omega_{eucl},
\end{equation}
so the metric ball $B_{\omega_f}(0,\frac{1}{4}r_{\Lambda})$ in $(B^n_{r_\Lambda},\omega_f)$ is relatively compact. Therefore the metric ball $B(p,\frac{1}{4}r_{\Lambda})$ in $(M,\omega)$ is also relatively compact.
The argument above tells that $B(p,\frac{1}{4}r_{\Lambda})$ is relatively compact for any $p\in M$, so $(M,\omega)$ is complete.
\end{proof}

\subsection{Proof of Lemma \ref{sc4}}
By the definition of extremal K\"ahler metrics, Lemma \ref{sc4} is a direct corollary of the fact
\begin{lem}\label{sc3}
Let $(M, \omega)$ be a connected $n$-dimensional K\"ahler manifold which admits a holomorphic isometric immersion $\phi$ into $ (\CP^N,\omega_{FS})$, $h$ a smooth real-valued function on $M$ and $X_h$ the $(1,0)$ part  the of $\nabla h$. Then $X_h$ is holomorphic if  and only if there is an $(N+1)$-order Hermitian matrix $A$ such that $X_h$ is $\phi$-related with $X^{[A]}$ (and $h=\frac{w^*Aw}{|w|^2}\circ\phi$).
\end{lem}

\begin{rmk} It is well-known that a real-valued smooth function $h$ on $(\CP^N,\omega_{FS})$ has holomorphic gradient if and only if $h=\frac{w^*Aw}{|w|^2}$ for some $(N+1)$-order Hermitian metrix $A$. Furthermore, when the equality holds, the $(1,0)$-part of $\nabla h$ is exactly $X^{[A]}$.
\end{rmk}

\begin{proof}[Proof of Lemma \ref{sc3}]
The ``if" part can be easily verified.  We only need to prove the ``only if" part.

Without loss of generality, we can assume that $\phi$ is a full immersion. Given a $p\in M$, we can find a connected open neighborhood $U$ of $p$, such that on $U$ there is a complex coordinate system $\{z^1,\cdots,z^n\}$ and $\phi$ can be expressed as $\phi(q)=[F(q)]$ for some holomorphic map $F:U\rightarrow \mathbb C^{N+1}$. As $\phi$ is full, $F$ is linearly full.
In the following we restrict our discussion on  $U$. First we have 
\begin{equation}
\omega={\textstyle\frac{\iu}{2}}\pbp\log |F|^2.
\end{equation}
Let $(g_{\bar j i})$ be the matrix defined by $\omega=\frac{\iu}{2}\sum\limits_{i,j=1}^n g_{\bar j i}dz^i\wedge d\bar z^j$ and $(g^{i\bar j})=(g_{\bar j i})^{-1}$.
Then we have
\begin{equation}
X_h=\sum_{i,j=1}^ng^{i\bar j}\frac{\partial h}{\partial \bar z^j}\frac{\partial}{\partial z^i}.
\end{equation}
Since $X_h$ is holomorphic, we have for any $i,l=1,\cdots n$
\begin{equation}
\frac{\partial}{\partial \bar z^l}\bigg(\sum_{j=1}^ng^{i\bar j}h_{\bar j}\bigg)=0,
\end{equation}
and consequently for any $l=1,\cdots,n$
\begin{equation}
\frac{\partial (X_{h}\log |F|^2)}{\partial \bar z^l}=\sum_{i,j=1}^ng^{i\bar j}\frac{\partial h}{\partial \bar z^j}\frac{\partial ^2 \log |F|^2}{\partial z^i\partial \bar z^l}=\sum_{i,j=1}^n g^{i\bar j}\frac{\partial h}{\partial \bar z^j}g_{\bar l i}=\frac{\partial h}{\partial \bar z^l}.
\end{equation}
Consider the function
\begin{equation}
\tilde h=X_{h}\log |F|^2-h.
\end{equation}
Then $\tilde h$ is holomorphic and
\begin{equation}
|F|^2h=F^*\big[\big(X_h-\tilde h\big)F\big].
\end{equation}
Since $\big(X_h-\tilde h\big)F$ is holomorphic and $F^*\big[\big(X_h-\tilde h\big)F\big]$ is real-valued, applying  Lemma \ref{rgdhol}, we can find an $(N+1)$-order Hermitian matrix $A$, such that
\begin{equation}\label{sceq11}
\big[\big(X_h-\tilde h\big)]F=AF,
\end{equation}
and henceforth
\begin{equation}
h=\frac{F^*AF}{|F|^2}=\frac{w^*Aw}{|w|^2}\circ\phi.
\end{equation}
By the canonical identification $T^{1,0}_{[w]}\CP^{N+1}\simeq \mathbb C^{N+1}/[w]$ and \eqref{sceq11}, one can easily check that $X_h$ is $\phi$-related with $X^{[A]}$.

By the above discussion, we know that $h$ is real-analytic and there is an $(N+1)$-order Hermitian matrix $A$, such that $X_h|_U$ is $\phi$-related with $X^{[A]}$ and $h|_U=\frac{w^*Aw}{|w|^2}\circ \phi$ for some nonempty open set $U$ of $M$. By the uniqueness of real-analytic and complex-analytic maps, together with the connectedness of $M$, it holds that $X_h$ is $\phi$-related with $X^{[A]}$ and $h=\frac{w^*Aw}{|w|^2}\circ \phi$.
\end{proof}

\subsection{Proof of Lemma \ref{lcext}}

First we give an expression for the scalar curvature of  a K\"ahler submanifold of $(\CP^N,\omega)$.

\begin{lem}\label{sc2}
Let $N>n\geq 1$, $U\subset \mathbb C^n$ an open domain, $f:U\rightarrow \mathbb C^{N-n}$ a holomorphic map and 
\begin{equation}
\omega={\textstyle\frac{\iu}{2}}\pbp\log (1+|z|^2+|f|^2).
\end{equation}
Then there exists a  holomorphic map $\Xi:U\rightarrow \mathbb C^{L(n,N)}$ with $L(n,N)\geq 1$, such that
\begin{equation}\label{sceq02}
n(n+1)-{\textstyle\frac{1}{4}}R(\omega)=\frac{|F|^4(|D^2f|^2+|\Xi|^2)}{\big|\frac{\partial F}{\partial z^1}\wedge\cdots\wedge \frac{\partial F}{\partial z^n}\big|^6},
\end{equation}
where $F=(1,z,f)$. Furthermore, the components of $\Xi$ can be written as polynomials of variables $z$ and partial derivatives of $f$ up to order $2$ of degree $3n+2$.
\end{lem}

\begin{proof}First we calculate $R(\omega)$ for an arbitrary K\"ahler metric $\omega$ which can be written as 
\begin{equation}
\omega={\textstyle\frac{\iu}{2}}\pbp\log |F|^2,
\end{equation}
for some holomorphic map $F:U\rightarrow \mathbb C^{N+1}$. Then we let $F=(1,z,f)$ to get \eqref{sceq02}.

To calculate $R(\omega)$, we notice that
\begin{equation}\label{scgl}
-{\textstyle\frac{1}{4}}R(\omega)=\Delta_\omega \log\det(g_{\bar j i}),
\end{equation}
where $\Delta_\omega=\sum\limits_{i,j=1}^ng^{i\bar j}\frac{\partial ^2}{\partial z^i \partial \bar z^j}$, $(g^{i\bar j})=(g_{\bar j i})^{-1}$ and $(g_{\bar j i})$ is the matrix-valued function defined by
\begin{equation}
\omega={\textstyle\frac{\iu}{2}}\sum_{i,j=1}^ng_{\bar j i} dz^i\wedge d\bar z^j.
\end{equation}
First we calculate $\det(g_{\bar j i})$. Let $F_i=\frac{\partial F}{\partial z^i}$ for $i=1,\cdots,n$, then we have
\begin{equation}
g_{\bar j i}=\frac{\partial ^2 (\log |F|^2)}{\partial z^i \partial \bar z^j}=|F|^{-2}(F_j^*F_i-|F|^{-2}F_j^*FF^*F_i),
\end{equation}
and henceforth
\begin{equation}\label{sceq03}
\det(g_{\bar j i})=|F|^{-2(n+1)}\left|\begin{array}{cc}
F^*F & F^*F_i \\
F_j^*F &F_j^*F_i
\end{array}\right|\\
=|F|^{-2(n+1)}|F\wedge F_1\wedge \cdots\wedge F_n|^2.
\end{equation}
Clearly $\Delta_\omega \log |F|^2=n$, so by \eqref{scgl}
\begin{equation}\label{scgl1}
n(n+1)-{\textstyle\frac{1}{4}}R(\omega)=\Delta_\omega \log |F\wedge F_1\wedge\cdots\wedge F_n|^2.
\end{equation}
Next we calculate $\Delta_\omega \log |F\wedge F_1\wedge\cdots\wedge F_n|^2$. Let $G:U\rightarrow \mathbb C^m$ be a holomoprhic map, $G_i=\frac{\partial G}{\partial z^i}$ for $i=1,\cdots,n$. When $\tau\in\mathbb C\rightarrow 0$, we have
\begin{equation}\det\left(g_{\bar j i}+|\tau|^2|F|^{-2}|G|^2\frac{\partial ^2 (\log |G|^2)}{\partial z^i \partial \bar z^j}\right)
=\det(g_{\bar j i})(1+|\tau|^2|F|^{-2}|G|^2\Delta_\omega\log |G|^2)+O(|\tau|^4).
\end{equation}
At the same time, we have
\begin{equation}\begin{aligned}
&\det\left(g_{\bar j i}+|\tau|^2|F|^{-2}|G|^2\frac{\partial ^2 (\log |G|^2)}{\partial z^i \partial \bar z^j}\right)\\
=&|F|^{-2(n+1)}\det\left(F_j^*F_i-|F|^{-2}F_j^*FF^*F_i+|\tau|^2(G_j^*G_i-|G|^{-2}iG_j^*GG^*G_)\right)\\
=&|F|^{-2(n+1)}|G|^{-2}\left|\begin{array}{ccc}
F^* F& 0 & F^*F_i\\
0 & GG^* & \tau G^*G_i\\
F_j^*F & \bar\tau G_j^*G & F_j^*F_i+|\tau|^2 G_j^*G_i
\end{array}\right|\\
=&|F|^{-2(n+1)}|G|^{-2}|(F,0)\wedge(0,G)\wedge(F_1,\tau G_1)\wedge\cdots\wedge(F_n,\tau G_n)|^2.
\end{aligned}
\end{equation}
Then we obtain
\begin{equation}\label{sceq04}
\Delta_\omega\log |G|^2=\frac{|F|^{-2n}}{\det(g_{\bar j i})|G|^4}\sum_{\begin{subarray}{c}\alpha_1<\dots<\alpha_n\\\beta_1<\beta_2\end{subarray}}
\left|\begin{array}{ccccc}
f^{\alpha_1} & 0 & \frac{\partial f^{\alpha_1}}{\partial z^1}&\cdots & \frac{\partial f^{\alpha_1}}{\partial z^n}\\
\vdots& \vdots  & \vdots &  & \vdots\\
f^{\alpha_n}& 0 & \frac{\partial f^{\alpha_n}}{\partial z^1}&\cdots & \frac{\partial f^{\alpha_n}}{\partial z^n}\\
0 & g^{\beta_1} & \frac{\partial g^{\beta_1}}{\partial z^1}&\cdots & \frac{\partial g^{\beta_1}}{\partial z^n}\\
0 & g^{\beta_2} & \frac{\partial g^{\beta_2}}{\partial z^1}&\cdots & \frac{\partial g^{\beta_2}}{\partial z^n}  
\end{array}\right|^2.
\end{equation}
By applying \eqref{sceq04} to $G=F\wedge F_1\cdots\wedge F_n$ and using \eqref{scgl1}, we arrive
\begin{equation}\label{sceq01}
n(n+1)-{\textstyle\frac{1}{4}}R(\omega)=\frac{|F|^4|\Xi_0|^2}{\big|\frac{\partial F}{\partial z^1}\wedge\cdots\wedge \frac{\partial F}{\partial z^n}\big|^6},
\end{equation}
where $\Xi_0:U\rightarrow \mathbb C^{L_0(n,N)}$ is a holomorphic map with $L_0(n,N)\geq 1$.

Now let $F=(1,z,f)$. By tedious but easy calculation, one can check that
\begin{equation}
|F\wedge F_1\wedge\cdots\wedge F_n|^2=1+|Df|^2+|h|^2,
\end{equation}
for some holomorphic map $h:U\rightarrow \mathbb C^{m_1}$ with $m_1\geq 1$. Then by \eqref{sceq04}, the map $\Xi_0$ in \eqref{sceq01} satisfies
\begin{equation}
|\Xi_0|^2=|D^2 f|^2+|\Xi|^2,
\end{equation}
for some holomorphic map $\Xi$ as mentioned in the Lemma. Consequently we obtain \eqref{sceq02}.
\end{proof}

\begin{proof}[Proof of Lemma \ref{lcext}]
Let $F=(1,z,f)$. For $i,j=1,\cdots, n$, we denote $F_i=\frac{\partial F}{\partial z^i}$, $f_i=\frac{\partial F}{\partial z^i}$ and $f_{ij}=\frac{\partial^2 f}{\partial z^i\partial z^j}$. Applying Lemma \ref{sc2} and using \eqref{sceqA1}, we have
\begin{equation}
\frac{|F|^2|(|D^2f|^2+|\Xi|^2)}{|F\wedge F_1\wedge\cdots\wedge F_n|^6}=\frac{F^*AF}{|F|^2},
\end{equation}
where $\Xi:B^n_r\rightarrow \mathbb C^{ L(n,N)}$ is a holomorprhic map with $L(n,N)\geq 1$.

First we prove the estimate \eqref{exteq1}. Since $A\leq \Lambda I_N$, we have $F^*AF\leq \Lambda|F|^2$. Then
\begin{equation}
|F|^2|D^2 f|^2\leq \Lambda |F\wedge F_1\wedge\cdots\wedge F_n|^6.
\end{equation}
At the same time
\begin{equation}
|F\wedge F_1\wedge\cdots\wedge F_n|^2\leq |F|^2\prod_{i=1}^n(1+|f_i|^2)\leq |F|^2(1+|Df|^2)^{n},
\end{equation}
so we have
\begin{equation}
|D^2f|^2\leq \Lambda (1+|z|^2+|f|^2)^2(1+|Df|^2)^{3n}.
\end{equation}
Given any $z_0\in\mathbb C^n$ with $|z_0|=1$, consider function
\begin{equation}
h(t)=(1+|z|+|f|^2+|Df|^2)^{\frac{1}{2}}\big|_{tz_0},
\end{equation}
we have
\begin{equation}
h'(t)\leq(1+|Df|^2+|D^2f|^2)^{\frac{1}{2}}\big|_{tz_0}\leq (h(t)^2+\Lambda h(t)^{6n+4})^{\frac{1}{2}}\leq (\Lambda+1)^{\frac{1}{2}}h(t)^{3n+2}.
\end{equation}
Since
\begin{equation}
h(0)=0,\qquad h'(t)\leq (\Lambda+1)^{\frac{1}{2}}h^{3n+2},
\end{equation}
we have
\begin{equation}
h(t)\leq (1-r_\Lambda^{-1}t)^{\frac{1}{3n+1}},
\end{equation}
where $r_\Lambda^{-1}=(3n+2)(\Lambda+1)^{\frac{1}{2}}$ as mentioned in the Lemma.

Next we prove the extensibility of $f$. Via diagonalizing  $A$, we can always find $(N+1)$-order matrices $A_1$ and $A_2$ such that $A=A_1^*A_1-A_2^*A_2$. Then we have
\begin{equation}
|F|^4(|D^2f|^2+|\Xi|^2)=(|A_1F|^2-|A_2F|^2)|F\wedge F_1\wedge\cdots\wedge F_n|^6,
\end{equation}
and consequently
\begin{equation}
(1+|z|^2+|f|^2)^2(|D^2f|^2+|\Xi|^2)+|FA_2|^2|F\wedge F_1\wedge\cdots\wedge F_n|^6=|FA_1|^2|F\wedge F_1\wedge\cdots\wedge F_n|^6,
\end{equation}
Apply Lemma \ref{rgdhol}, for each $\alpha=1,\cdots, k$ and $i,j=1,\cdots,n $, we can find holomorphic polynomials $P^\alpha_{ij}$ of degrees no more than $3n+4$ on $\mathbb C^N\times (\mathbb C^{N-n}\otimes\mathbb C^n)$, such that
\begin{equation}
f^{\alpha}_{ij}=P^\alpha_{ij}(z,f,Df)
\end{equation}
Let $\xi=(z,f,Df)$, then $\xi$ solves
\begin{equation}
D\xi=\mathcal A(\xi),
\end{equation}
where $\mathcal A (u,v,B)=(\mathrm{I}_n,B,P(u,v,B))$ and $P:\mathbb C^{N+(N-n)n}\rightarrow \mathbb C^{N-n}\otimes\mathbb C^n\otimes\mathbb C^n$ is the map whose $(\alpha,i,j)$ component is $P^\alpha_{ij}$. Clearly we have
\begin{enumerate}[label=(\alph*)]
\item $\mathcal A$ is defined on whole $\mathbb C^{n+k+nk}$,
\item $1+|\xi|^2\leq (1-r_\Lambda^{-1}|z|)^{\frac{2}{3n+2}}$ if $\xi$ is defined on $B^n_{\tilde r}$ with $\tilde r\leq r_\Lambda$.
\end{enumerate}
By Proposition \ref{extholo}, $\xi$ can be extended to a holomorphic map on $B^n_{\tilde r}$ with some $\tilde r\geq r_\Lambda$.
\end{proof}

\subsection{A remark on Theorem \ref{mthm2}}
\begin{prop} Let $(M,\phi)$ be as in Theorem \ref{mthm2}. If $p_1$ and $p_2$ are two distinct point in $M$, then either $\phi(p_1)\neq\phi(p_2)$ or $(d\phi)_{p_1}(T^{1,0}M)\neq (d\phi)_{p_2}(T^{1,0}M)$.
\end{prop}
\begin{proof}We need to show that if $\phi(p_1)=\phi_2(p_2)$ and $(d\phi)_{p_1}(T^{1,0}M)=(d\phi)_{p_2}(T^{1,0}M)$ for some $p_1,p_2\in M$, then $p_1=p_2$.

Following the proofs of Theorem \ref{mthm2} and Lemma \ref{lcext} we can find a holomorphic chart $(U_1,\varphi_1)$ around $p_1$ such that $\varphi_1(p_1)=0$ and
\begin{equation}\label{d2feq1}
\phi\circ\varphi_1^{-1}(z)=[Q(1,z,f_1)], \qquad D^2f_1=P(z,f_1,Df_1),
\end{equation}
where $Q$ is $(N+1)$-order unitary matrix, $f:\varphi_1(U_1)\rightarrow \mathbb C^{N-n}$ is a holomorphic map with $f_1(0)=0$ and $Df_1(0)=0$ and $P:\mathbb C^N\times (\mathbb C^{N-n}\otimes\mathbb C^n)\rightarrow \mathbb C^{N-n}\otimes\mathbb C^n\otimes \mathbb C^n$ whose components are all polynomials.

Let $\varphi=(\varphi^1,\cdots,\varphi^n)$ and $(h^1,\cdots,h^{N-n})$ be maps with mereomorphic components defined by
\begin{equation}
(\varphi,h)=\frac{Q^*w}{w_1^*w}\circ\phi,
\end{equation}
where $w_1$ is the first column vector of $Q$. For $\alpha=1,\cdots,N-n$ and $i,j=1,2,\cdots,n$, define mereomorphic functions
\begin{equation}
h^{\alpha}_i=\frac{dh^\alpha\wedge\theta_i}{\theta},\qquad h^{\alpha}_{ij}=\frac{dh^\alpha_i\wedge\theta_j}{\theta},
\end{equation}
where
\begin{equation}
\theta_i=(-1)^{i-1} d\varphi^1\wedge\cdots \wedge d\varphi^{i-1}\wedge d\varphi^{i+1}\wedge\cdots\wedge d\varphi^n,\qquad \theta=d\varphi^1\wedge\cdots\wedge d\varphi^n.
\end{equation}
By \eqref{d2feq1}, we have $\varphi_1=\varphi|_{U_1}$, $f_1\circ\varphi_1=h|_{U_1}$ and consequently
\begin{equation}\label{d2heq}
(h^\alpha_{ij})=P(\varphi,h,(h^\beta_k)).
\end{equation}

By the conditions $\phi(p_1)=\phi_2(p_2)$ and $(d\phi)_{p_1}(T^{1,0}M)=(d\phi)_{p_2}(T^{1,0}M)$, one can easily check that
\begin{equation}
\varphi(p_2)=0,\qquad (d\varphi^1\wedge\cdots\wedge d\varphi^n)|_{p_2}\neq 0,\qquad h(p_2)=0,\qquad  dh|_{p_2}=0.
\end{equation}
Then there exists an open neighborhood $U_2$ of $p_2$, such that $\varphi_2=\varphi|_{U_2}$ is a holomorphic chart map. Let $f_2=h\circ\varphi_2^{-1}$, then $f_2(0)=0$ and $Df_2(0)=0$.  Furthermore, by the definition of $(h^\alpha_i)$ and $(h^\alpha_{ij})$ together with \eqref{d2heq}, we have
\begin{equation}\label{d2feq2}
D^2f_2=P(z,f_2,Df_2).
\end{equation}
Recall that $f_1(0)=0$, $Df_1(0)=0$ and $D^2f_1=P(z,f_1,Df_1)$, we have $f_1=f_2$ on  $B^n_r$ for some $r>0$.
Since we have two holomorphic immersion $\varphi_1^{-1},\varphi_2^{-1}:B^n_r\rightarrow M$ with $\phi\circ\varphi_1^{-1}=\phi\circ\varphi_2^{-1}$, the definition of complex leaf pairs indicates that $\varphi_1^{-1}=\varphi_2^{-1}$ and consequently $p_1=\varphi_1^{-1}(0)=\varphi_2^{-1}(0)=p_2$.
\end{proof}

\section{Further remarks on complete extremal K\"ahler submanifolds}\label{lccs}

\subsection{Bochner coordinates at a critical point} Inspired by Lemma \ref{sc4}, we consider the class of full holomorphic isometric immersions $\phi:(M,\omega)\rightarrow (\CP^N,\omega_{FS})$ which satisfy
\begin{itemize}
\item $(M,\omega)$ is a connected $n$-dimensional K\"ahler manifold
\item There exists a nonzero holomorphic vector field $X$ on $M$ and an $(N+1)$-order Hermitian matrix $A$ such that $X$ is $\phi$ related to $X^{[A]}\in H^0(\CP^N,T^{1,0}\CP^N)$.
\end{itemize} 
We have the following result about the zero points of $X$, namely the critical points of $\frac{w^*Aw}{|w|^2}\circ\phi$
\begin{lem}\label{exstcrt}
$X_p=0$ for some $p\in M$ if and only if $\phi(p)=[w_p]$ for some eigenvector $w_p$ of $A$.

Furthermore, assume that  $(M,\omega)$ is complete, and let $E_{max}$ and $E_{min}$ be the eigenspaces of $A$ corresponding to the maximal and minimal eigenvalues of $A$ respectively. Then both $\phi^{-1}(\mathrm P(E_{max}))$ and $\phi^{-1}(\mathrm P(E_{min}))$ are nonempty.
\end{lem}
\begin{proof}
It is well-known that for any $w\in \mathbb C^N\!\setminus\!\{0\}$, $X^{[A]}_{[w]}=0$ if and only if $[w]$ is an eigenvector of $A$. Since $X$ is $\phi$-related with $X^{[A]}$ and $\phi$ is a holomorphic immersion, $X_p=0$ if and only if $d\phi_p(X_p)=X^{[A]}_{\phi(p)}=0$. Thus $X_p=0$ if and only if $\phi(p)=[w_p]$ for some eigenvector $w_p$ of $A$.

Now assume that $(M,\omega)$ is complete. We have $|X|\leq \sup_{\CP^N}|X^{[A]}|<\infty$. Henceforth $X$ generates a global holomorphic flow $\phi^X:\mathbb C\times M\rightarrow M$.
Let $\phi^{[A]}:\mathbb C\times\CP^N\rightarrow\CP^N$ be the holomorphic flow on $\CP^N$ generated by $X^{[A]}$. Since $X$ is $\phi$-related to $X^{[A]}$, we have
\begin{equation}
\phi(\phi^X(\tau,p))=\phi^{[A]}(\tau,\phi(p)),\qquad \forall (\tau,p)\in\mathbb C\times M.
\end{equation}
For any $p\in M$ with $\phi(p)=[w_p]$ and $t_2>t_1$, we can estimate $d(\phi^X(t_2,p),\phi^X(t_1,p))$ by calculating the length of the curve $c_p(t)=\phi^X(p,t)$ on $[t_1,t_2]$, which is equal to the lenght of the curve $c_{[w_p]}(t)=[e^{tA}w_p]$ on $[t_1,t_2]$. Especially, since  $\phi$ is full, we can find $p\in M$ such that $\phi(p)=[w_1+w_{R}]$ for some $w_1\in E_{max}\!\setminus\!\{0\}$ and $w_{R}\in E_{max}^{\perp}$.  One can easily check that for any $t_2>t_1$
\begin{equation}
d(\phi^X(t_2,p),\phi^X(t_1,p))\leq \frac{(\mu_1-\mu_l)|w_{R}|}{(\mu_1-\mu_2)|w_1|}(e^{-(\mu_1-\mu_2)t_2}-e^{-(\mu_1-\mu_2)t_1}),
\end{equation}
where $\mu_1>\mu_2>\cdots>\mu_l$ are distinct eigenvalues of $A$.
Noting that $(M,\omega)$ is complete, when $t\rightarrow\infty$, $\phi^X(t,p)$ converges to some $p_{+}\in M$ with $\phi(p_{+})=[w_1]$.

Similarly, we can find a $p_{-}\in M$ with $\phi(p_{-})=[w_l]$ for some $w_l\in E_{min}\!\setminus\!\{0\}$.
\end{proof}

\begin{lem}\label{lccscrt}
 If $X(p)=0$ for some $p\in M$, then we can find a neighborhood $U$ of $p$, a holomorphic chart map $\varphi:U\rightarrow \mathbb C^n$ and an $(N+1)$-order unitary matrix $Q$, real numbers $a,b^1,\cdots,b^n, c^1,\cdots,c^{N-n}$ and a holomorphic map  $f=(f^1,\cdots,f^{N-n}):\varphi(U)\rightarrow \mathbb C^n$, such that
\begin{enumerate}
\item $\varphi(p)=0$, $f(0)=0$, $Df(0)=0$ and
\begin{equation}
\phi\circ\varphi^{-1}(z)=[Q(1,z,f(z))].
\end{equation}
\item $Q^*AQ-aI_{N+1}=\mathrm{diag}(0,b^1,\cdots,b^n,c^1,\cdots,c^{N-n})$.
\item $\varphi_* X=\sum\limits_{i=1}^n b^iz^i\frac{\partial}{\partial z^i}$ and $(\varphi_* X)f^\alpha=c^\alpha f^\alpha$ for each $\alpha=1,\cdots,N-n$.
\end{enumerate}
\end{lem}

\begin{proof}
By choosing a Bochner coordinate system at $p$, we can find a neighborhood $U$ of $p$, a holomoorphic chart map $\varphi:U\rightarrow\mathbb C^n$,  an $(N+1)$-order unitary  matrix $Q$ and a holomorphic map  $f=(f^1,\cdots,f^{N-n}):\varphi(U)\rightarrow \mathbb C^n$,  such that $\varphi(p)=0$, $f(0)=0$, $Df(0)=0$ and
\begin{equation}
\phi\circ\varphi^{-1}(z)=[Q(1,z,f(z))].
\end{equation}
Since $X_p=0$, the first column vector $w_1$ of $Q$, which satisfies $\phi(p)=[w_1]$, is a eigenvalue of $A$. Let $a$ be the eigenvalue of $A$ corresponding to $w_1$. Namely $Aw_1=aw_1$. Together with the facts that $A$ is Hermitian and $Q$ is unitary, we can write
\begin{equation}\label{diagA1}
Q^*AQ-a I_{N+1}=\mathrm{diag}(0,A_0),
\end{equation}
where $A_0$ is an $N$-order Hermitian matrix.

Since $X$ is $\phi$-related to $X^{[A]}$ and $\phi=[Q(1,\varphi,f\circ\varphi)]$, we can find a holomorphic function $\xi$ on $U$ such that
\begin{equation}
XQ(1,\varphi,f\circ\varphi)=(\xi I_{N+1}+A)Q(1,\varphi,f\circ\varphi),
\end{equation}
namely
\begin{equation}
X(1,\varphi,f\circ\varphi)=((\xi+a)I_{N+1}+Q^*AQ-a I_{N+1})(1,\varphi,f\circ\varphi).
\end{equation}
Using \eqref{diagA1} and noting that $X1=0$, we have $\xi\equiv-a$ and
\begin{equation}\label{related1}
(X\varphi,X(f\circ\varphi))=A_0(\varphi,f\circ\varphi).
\end{equation}
Let $\tilde X=\varphi_* X$. Then \eqref{related1} can be written as 
\begin{equation}
(\tilde X z,\tilde X f)=A_0(z,f).
\end{equation}
We write $A_0$ as
\begin{equation}
A_0=\left(\begin{array}{cc}
A_1 & B^*\\
B & A_2
\end{array}\right),
\end{equation}
where $A_1$ and $A_2$ are Hermitian matrices  of orders $n$ and $N-n$ respectively. Then we have 
\begin{equation}
\tilde X f=Bz+A_2 f.
\end{equation}
When $z\rightarrow 0$, $f(z)=O(|z|^2)$ and $\tilde X f(z)=O(|z|^2)$. So we have $B=0$ and consequently
\begin{equation}
A_0=\mathrm{diag}(A_1,A_2).
\end{equation}
We choose unitary matrices $Q_1$ and $Q_2$ of orders $n$ and $N-n$ respectively  such that
\begin{equation}
Q_1^*A_1Q_1=\mathrm{diag}(b^1,\cdots,b^n),\qquad Q_2^*A_2Q_2=\mathrm{diag}(c^1,\cdots,c^{N-n}).
\end{equation}
Then we have
\begin{equation}\label{related2}
\tilde X (Q_1^* z,Q_2^* f)=(Q_1^*A_1Q_1(Q_1^*z),Q_2^*A_2Q_2(Q_2^*f)),
\end{equation}
Let $T:\mathbb C^n\rightarrow \mathbb C^n$ be th map defined by $T(z)=Q_1z$ and
\begin{equation}
\hat\varphi=T^{-1}\circ\varphi,\qquad \hat Q=Q\mathrm{diag}(0,Q_1,Q_2),\qquad \hat f=(\hat f^1,\cdots,\hat f^{N-n})=Q_2^*f\circ T.
\end{equation}
Then $\hat\varphi$ is a holomorphic chart map, $\hat Q$ is unitary, $\hat f$ is holomorphic,  and
\begin{enumerate}[label=(\arabic*')]
\item $\hat\varphi(p)=0$, $\hat f(0)=0$, $D\hat f(0)=0$ and
\begin{equation}
\phi\circ\hat\varphi^{-1}(z)=[\hat Q(1,z,\hat f(z))].
\end{equation}
\item $\hat Q^*A\hat Q-aI_{N+1}=\mathrm{diag}(0,b^1,\cdots,b^n,c^1,\cdots,c^{N-n})$.
\end{enumerate}
Moreover, let $\hat X=\hat\varphi_* X=(T^{-1})_* \tilde X$, then by \eqref{related2}
\begin{equation}
(\hat X z,\hat X f)=(Q_1^*A_1Q_1 z,Q_2^*A_2Q_2 \hat f),
\end{equation}
and consequently we have
\begin{enumerate}[label=(\arabic*')]\setcounter{enumi}{2}
\item $\hat X=\sum\limits_{i=1}^nb^iz^i\frac{\partial}{\partial z^i}$ and $\hat X \hat f^{\alpha}=c^\alpha \hat f^{\alpha}$ for each $\alpha=1,\cdots,N-n$.
\end{enumerate}
This concludes the proof.
\end{proof}

As a complement of  Lemma \ref{lccscrt}, we have the following easy fact
\begin{lem}\label{lccscrt1}
In Lemma \ref{lccscrt}, for $\alpha=1,\cdots,N-n$, the Taylor expansion of $f^\alpha$ at $0$ has the form
\begin{equation}
f^\alpha=\sum_{b^1i_1+\cdots+b^ni_n=c^\alpha}a_Iz^I,
\end{equation}
where the  multi-index $I=(i_1,\cdots,i_n)$. Especially if $b^1,\cdots,b^n$ are all positive (resp. negative) , then each $f^\alpha$ is a polynomial.
\end{lem}

\subsection{Proof of Theorem \ref{thm2}}
Now we are in position to prove Theorem \ref{thm2}, which is
\begin{thm}
Let $(S,\omega_S)$ be a connected extremal  K\"ahler hypersurface of  $(\CP^{n+1},\omega_{FS})$. Then the scalar curvature $R(\omega_S)$ is non-constant only if $n\geq 3$ and $R(\omega)\leq 4n(n+1)-8$.
\end{thm}
\begin{proof}
Let $(M,\phi)$ be the leaf extension of $S$ and $\omega=\phi^*\omega_{FS}$. We only need to show that $R(\omega)$ is non-constant only if $n\geq 3$ and $R(\omega)\leq 4n(n+1)-8$.

Applying Theorem \ref{mthm1}, $(M,\omega)$ is complete. Moreover, clearly $\omega$ is extremal, so by  Lemma \ref{sc4}, there exists a holomorphic $(1,0)$-vector field $X$ on $M$ and $(n+2)$-order Hermitian matrix $A$, such that
 $X$ is $\phi$ related with $X^{[A]}$ and
\begin{equation}
n(n+1)-\frac{1}{4}R(\omega)=\frac{w^*Aw}{|w|^2}\circ \phi.
\end{equation}
We only need to consider the case that $R(\omega)$ is non-constant. This can happen only  if $\phi$ is full and $A$ is not a scalar matrix. Let $\mu_{max}$ and $\mu_{min}$ be the minimal and maximal eigenvalues of $A$ respectively. And let $E_{max}$ and $E_{min}$ be the eigenspaces of $A$ corresponding to the $\mu_{max}$ and $\mu_{min}$ respectively. By Lemma \ref{exstcrt}, both $Z_{max}=\phi^{-1}(\mathrm P(E_{max}))$ and $Z_{min}=\phi^{-1}(\mathrm P(E_{min}))$ are nonempty.

First we show that $\dim E_{min}\geq 2$.
Suppose oppositely $\dim E_{min}=1$. By choosing a $p\in Z_{min}$ and applying Lemma \ref{lccscrt} and \ref{lccscrt1},  we can reduce to the case that $(M,\phi)$ is the leaf extension of
\begin{equation}
\{[1,z,f(z)]|z\in B^n_r\},
\end{equation}
where $r>0$ and $f$ is a nonzero holomoprhic polynomial on $\mathbb C^n$ with $f(0)=0$ and $Df(0)=0$. Using the fact that $(M,\omega)$ is complete and Proposition \ref{qnleaf}, we deduce that $M$ is biholomorphic to $Q_n$ and $\phi$ is an embedding. Noting that $H^2(Q_n,\mathbb Z)=\mathbb Z$, the canonical K\"ahler metric on $Q_n$ is Einstein and $\omega$ is extremal, the uniqueness of extremal K\"ahler metrics (\cite{ChT05}) indicates that $\omega$ is Einstein.  Which is contradict to the condition that $R(\omega)$ is non-constant. So we must have $\dim E_{min}\geq 2$.

Similarly, we must have $\dim E_{max}\geq 2$.

Next we show that $\mu_{min}\geq 2$, $\dim E_{min}\geq 3$ and $\mu_{min}\geq 1$.
We choose a $p\in Z_{min}$. Applying  Lemma \ref{lccscrt}, we can  find  a neighborhood $U$ of $p$, a holomorphic chart map $\varphi:U\rightarrow \mathbb C^n$ and an $(n+2)$-order unitary matrix $Q$, nonnegative real numbers $b^1,\cdots,b^n, c$ and a holomorphic $f:\varphi(U)\rightarrow \mathbb C^{N-n}$, such that
\begin{enumerate}[label=(\alph*)]
\item $\varphi(p)=0$, $f(0)=0$, $Df(0)=0$ and
\begin{equation}
\phi\circ\varphi^{-1}(z)=[Q(1,z,f(z))].
\end{equation}
\item $b^1\leq b^2\leq\cdots\leq b^n$ and $Q^*AQ-\mu_{min}I_{n+2}=\mathrm{diag}(0,b^1,\cdots,b^n,c)$.
\item $\varphi_* X=\sum\limits_{i=1}^n b^iz^i\frac{\partial}{\partial z^i}$ and $(\varphi_* X)f=c f$.
\end{enumerate}
Following the proof of $\dim E_{min}\geq 2$, we can easily verify that $b^1=0$. Noting that $\dim E_{max}\geq 2$,  we have $\mu_{min}+b^n=\mu_{max}$, then $b^n>0$ and $c\leq b^n$. Let $z'=(z^1,\cdots,z^{n-1})$, we can expand $f$ as
\begin{equation}
f(z)=\bigg(\sum_{i=1}^{n-1}v_iz^i+\frac{1}{2}\sum_{i,j=1}^{n-1} a_{ij}z^iz^j+R_1(z')\bigg)z^n+\frac{1}{2}\sum_{i,j=1}^{n-1}b_{ij}z^iz^j+R_2(z'),
\end{equation}
where $R_1(z')=O(|z'|^3)$ and $R_2(z')=O(|z'|^3)$, when $z'\rightarrow 0$. Moreover, (c) and Lemma \ref{lccscrt1} imply
\begin{enumerate}[label=(\roman*)]
\item $v_i\neq 0$ only if $b^i=0$ and $c=b^n $;
\item $a_{ij}\neq 0$ only if $b^i=b^j=0$ and $c=b^n $;
\item $b_{ij}\neq 0$ only if $b^i+b^j=c$.
\end{enumerate}
By (i), up to a proper coordinate transposition, we can assume that $v_i=0$ for $i=2,\cdots,n-1$. Consider $\omega_f=\frac{\iu}{2}\pbp\log(1+|z|^2+|F|^2)(\varphi^{-1})^*\omega$. Using formula \eqref{sceq01}, we can obtain on $\{z'=0\}$
\begin{equation}\begin{aligned}
n(n+1)-\frac{1}{4}R(\omega_f)=&\frac{1+|z^n|^2}{(1+|v_1|^2|z^n|^2)^3}\bigg[2|v_1|^2(1+|v_1|^2|z^n|^2)(1+|z^n|^2)+|\ell_{11}(z^n)|^2\\
&+2(1+|v_1|^2|z^n|^2)\sum_{i=2}^{n-1}|\ell_{1i}(z^n)|^2+(1+|v_1|^2|z^n|^2)\sum_{i,j=2}^{n-1}|\ell_{ij}(z^n)|^2\bigg],
\end{aligned}\end{equation}
where $\ell_{ij}(z^n)=a_{ij}z^n+b_{ij}$ for $i,j=2,\cdots,n$.
At the same time, we have on $\{z'=0\}$
\begin{equation}
n(n+1)-\frac{1}{4}R(\omega_f)=\frac{\mu_{min}+\mu_{max}|z^n|^2}{1+|z^n|^2}.
\end{equation}
Noting that each $\ell_{ij}(z^n)$ is a linear function of $z^n$, we deduce that
\begin{equation}
|v_1|^2=1,\qquad a_{11}=b_{11}= 0, \qquad b_{ij}=0,\quad\forall i,j=2,\cdots,n-1.
\end{equation}
and
\begin{equation}
0=\sum_{i=2}^{n-1}a_{1i}\overline{b_{1i}},\qquad
\mu_{min}=2+2\sum_{i=2}^{n-1}|b_{1i}|^2,\qquad
\mu_{max}=2+2\sum_{i=2}^{n-1}|a_{1i}|^2.
\end{equation}
Since $\mu_{max}>\mu_{min}$, we must have
\begin{equation}
\sum_{i=2}^{n-1}|a_{1i}|^2>\sum_{i=2}^{n-1}|b_{1i}|^2\geq 0.
\end{equation}
By (ii), $a_{1i}\neq 0$ holds only if $b^i=0$ and $c=b^n$. Therefore we have $n\geq 3$ and $b^1=b^2=0$. Consequently $Q^*AQ-\mu_{min}I_{N+1}=\mathrm{diag}(0,0,0,b^3,\cdots,b^n,b^n)$ and $\dim E_{min}\geq3$.

Here we need to point out that if we start with $p\in Z_{max}$, then we will get $b_{1i}\neq0$ for some $i=2,\cdots,n-1$, which still implies $\dim E_{min}\geq 3$ rather than $\dim E_{max}\geq 3$.

\smallskip

By the argument above, if $R(\omega)$ is non-constant, we have $n\geq 3$ and
\begin{equation}
R(\omega)\leq 4(n(n+1)-\mu_{min})\leq 4n(n+1)-8.
\end{equation}
This concludes the proof.
\end{proof}


\begin{thebibliography}{99}


\bibitem{Boc47} S. Bochner,
{\em Curvature in Hermitian metric},
Bull. Amer. Math. Soc. {\bf 53} (1947), 179–195.


\bibitem{BR62} A. Borel, R. Remmert,
{\em Über kompakte homogene Kählersche Mannigfaltigkeiten},
Math. Ann. {\bf 145} (1961/62), 429–439.

\bibitem{Cal53} E. Calabi,
{\em Isometric imbedding of complex manifolds},
Ann. of Math. (2) {\bf 58} (1953), 1–23.


\bibitem{Cal82} E. Calabi,
{\em Extremal Kähler metrics},
Seminar on Differential Geometry, pp. 259–290,
Ann. of Math. Stud., {\bf 102}, Princeton Univ. Press, Princeton, N.J., 1982.

\bibitem{ChT05} X. Chen, G. Tian
{\em Uniqueness of extremal Kähler metrics}, 
C. R. Math. Acad. Sci. Paris {\bf 340} (2005), no. 4, 287–290.

\bibitem{Chern67} S.S. Chern,
{\em Einstein hypersurfaces in a Kählerian manifold of constant holomorphic curvature},
J. Differential Geometry {\bf 1} (1967), no. 1, 21–31.






\bibitem{Hano75} J. Hano,
{\em Einstein complete intersections in complex projective space},
Math. Ann. {\bf 216} (1975), no. 3, 197–208.


\bibitem{Hul96} D. Hulin,
{\em Sous-variétés complexes d'Einstein de l'espace projectif,} 
Bull. Soc. Math. France {\bf 124} (1996), no. 2, 277–298.

\bibitem{Hul00} D. Hulin,
{\em Kähler-Einstein metrics and projective embeddings},
J. Geom. Anal. {\bf 10} (2000), no. 3, 525–528.

\bibitem{Kob67} S. Kobayashi,
{\em Hypersurfaces of complex projective space with constant scalar curvature},
J. Differential Geometry {\bf 1} (1967), 369–370.

\bibitem{LSZ21} A. Loi, F. Salis, F. Zuddas,
{\em Extremal Kähler metrics induced by finite or infinite-dimensional complex space forms},
J. Geom. Anal. {\bf 31} (2021), no. 8, 7842–7865.



\bibitem{Smy67} B. Smyth,
{\em Differential geometry of complex hypersurfaces},
Ann. of Math. (2) {\bf 85} (1967), 246–266.


\bibitem{Tsu86} K. Tsukada,
{\em Einstein Kähler submanifolds with codimension 2 in a complex space form},
Math. Ann. {\bf 274} (1986), no. 3, 503–516.

\bibitem{Ume87} M. Umehara,
{\em Einstein Kaehler submanifolds of a complex linear or hyperbolic space},
Tohoku Math. J. (2) {\bf 39} (1987), no. 3, 385–389. 

\end{thebibliography}
\end{document}